\documentclass[11pt]
{amsart}
\usepackage{amssymb}

\input xy  \xyoption{all}

\usepackage[left=1.0in,right=1.0in,top=1.0in,bottom=1.0in]{geometry}

\newtheorem{deff}{Definition}[section]
\newtheorem{lemma}[deff]{Lemma}
\newtheorem{theorem}[deff]{Theorem}
\newtheorem{corollary}[deff]{Corollary}

\newtheorem{proposition}[deff]{Proposition}
\newtheorem{fact}[deff]{Fact}

\newtheorem{Claim}{Claim}

\theoremstyle{definition}
\newtheorem{example}[deff]{Example}
\newtheorem{definition}[deff]{Definition}
\newtheorem{question}[deff]{Question}
\newtheorem{remark}[deff]{Remark}

\def\ker{\mathop{\rm ker}}

\newcommand{\T}{{\mathbb T}}
\newcommand{\Z}{{\mathbb Z}}
\newcommand{\N}{{\mathbb N}}
\def\P{{\mathbb P}}
\newcommand{\cont}{\mathfrak c}

\newcommand{\AVp}{Arhangel'ski\u{\i} property}

\newcommand{\AV}{Arhangel'ski\u{\i}}
\newcommand{\sAV}{projectively \AV}
\newcommand{\B}{semi-Bernstein}

\def\grp#1{\langle{#1}\rangle}
\def\V{\mathcal{V}}
\def\supp#1#2{\mathrm{supp}_{#2}(#1)}

\title{Metrization criteria for compact groups in terms of their dense subgroups}

\author[D. Dikranjan]{Dikran Dikranjan}
\address[D. Dikranjan]{Dipartimento di Matematica e Informatica\\
Universit\`{a} di Udine\\
Via delle Scienze  206, 33100 Udine\\
Italy}
\email{dikran.dikranjan@uniud.it}

\author[D. Shakhmatov]{Dmitri Shakhmatov}
\address[D. Shakhmatov]{Division of Mathematics, Physics and Earth Sciences\\
Graduate School of Science and Engineering\\
Ehime University\\
Matsuyama 790-8577\\
Japan}
\email{dmitri.shakhmatov@ehime-u.ac.jp}

\begin{document}

\dedicatory{Dedicated to Professor A.~V.~Arhangel'ski\u{\i} on the occasion of his 73rd anniversary}

\thanks{The first named author was partially supported by SRA, grants P1-0292-0101 and J1-9643-0101, and by grant MTM2009-14409-C02-01.}

\thanks{The second named author was partially supported by the Grant-in-Aid for
Scientific Research~(C) No.~22540089 by the Japan Society for the Promotion of Science (JSPS)}


\begin{abstract} 
According to Comfort, Raczkowski and Trigos-Arrieta,
a dense subgroup $D$ of a compact abelian group
$G$ determines $G$ if the restriction homomorphism $\widehat{G}\to \widehat{D}$ of the dual groups is a topological isomorphism.
We introduce four conditions on $D$ that are necessary for it to determine $G$ and we resolve the following question:
If one of these conditions holds for every dense (or $G_\delta$-dense) subgroup $D$ of $G$, must $G$ be metrizable? 
In particular,
we prove (in ZFC) 
that a compact abelian group determined by all its $G_\delta$-dense subgroups is metrizable,
thereby resolving Question 5.12(iii) from
 [S.~Hern\' andez, S.~Macario and  F.~J.~Trigos-Arrieta, Uncountable products of determined groups need not be determined, 
{\sl J. Math. Anal. Appl.\/}  348 (2008), 834--842]. 
(Under the additional assumption of the Continuum Hypothesis CH,
the same statement was proved recently by
Bruguera, Chasco, Dom\'\i nguez, Tkachenko and Trigos-Arrieta.)
As a 
tool, we develop a 
machinery for building $G_\delta$-dense 
subgroups without uncountable compact subsets in compact groups of weight $\omega_1$
(in ZFC). 
The construction is 
delicate, as these subgroups must have non-trivial convergent sequences in some models of ZFC.
\end{abstract}

\maketitle

\bigskip
{\em All spaces and topological groups are assumed to be Hausdorff.\/} 
Recall that a topological space $X$ is called:
\begin{itemize}
\item
{\em $\kappa$-bounded\/} (for a given cardinal $\kappa$) if the closure of every subset of $X$ of cardinality at most $\kappa$ is compact,
\item
{\em countably compact\/} if every countable open cover of $X$ has a finite subcover,
\item
{\em pseudocompact\/} if every real-valued continuous function defined on $X$ is bounded.
\end{itemize}

It is well known that 
$$
\mbox{compact}
\to
\mbox{$\kappa$-bounded}
\to
\mbox{$\omega$-bounded}
\to
\mbox{countably compact}
\to
\mbox{pseudocompact}
$$
for every infinite cardinal $\kappa$.

Symbols $w(X)$, $nw(X)$ and $\chi(X)$ denote the weight, the network weight and the character of a space $X$, respectively.
All undefined topological terms can be found in \cite{E}.
 
As usual,
$\N$ denotes the set of natural numbers, $\P$ denotes the set of all prime numbers,
$\Z$ denotes the group of integers, $\Z(p)=\Z/p\Z$ denotes
the cyclic group of order $p\in\P$ with the discrete topology and $\T$ denotes the circle group with its usual topology.
The symbol
$\cont$ denotes the cardinality of the continuum,  $\omega_1$ denotes the first uncountable cardinal and $\omega=|\N|$. Clearly, $\omega<\omega_1$.
By Cantor's theorem, $\omega_1\le\cont$.
The Continuum Hypothesis CH says that $\omega_1=\cont$. We recall that this equality 
is both consistent with and independent of the usual Zermelo-Fraenkel axioms ZFC of set theory \cite{K}.  

Recall that a cardinal $\tau$ is {\em strong limit\/} if $2^\sigma<\tau$ for every cardinal $\sigma<\tau$.
For an ordinal (in particular, for a cardinal) $\alpha$, we denote by $\mathrm{cf}(\alpha)$ the cofinality of $\alpha$. For a cardinal $\kappa$ and a set $X$, the symbol $[X]^{\le\kappa}$ denotes the family of all subsets of set $X$ having cardinality at most $\kappa$. 
All undefined set-theoretic terms can be found in \cite{K}.

\section{Introduction}

Let $G$ be an abelian topological group. We denote by
$\widehat{G}$ the dual group of all continuous characters endowed with the compact-open topology. 
Following \cite{CRT,CRT1}, we say that a dense subgroup $D$ of 
$G$ {\em determines\/} $G$ if the restriction homomorphism $\widehat{G}\to \widehat{D}$ of the dual 
groups is a topological isomorphism. According to \cite{CRT,CRT1}, 
$G$ is said to be {\em  determined\/} if every dense subgroup of $G$ determines $G$.
The cornerstone in this topic is the following theorem due to Chasco and Au\ss enhofer:  

\begin{theorem}\label{CA} \cite{Diss,CM} Every  metrizable abelian group is  determined.  
\end{theorem}

A remarkable partial inverse of this theorem was proved by
Hern\'andez, Macario and Trigos-Arrieta.
(Under the assumption of the Continuum Hypothesis, this was established earlier by Comfort, Raczkowski and Trigos-Arrieta in  \cite{CRT, CRT1}).

\begin{theorem}\label{CHMRT} \cite[Corollary 5.11]{HMT} Every compact determined abelian group is metrizable. 
\end{theorem}

While Theorem \ref{CA} says that {\em every\/} dense subgroup of a metrizable abelian group determines it, Theorem \ref{CHMRT}  asserts
that every non-metrizable compact abelian group necessarily contains {\em some\/} dense subgroup that does not determine it.

A subgroup $D$ of a topological group $G$ is  called
{\em $G_\delta$-dense\/} in $G$ if $D\cap B\not=\emptyset$ for every non-empty $G_\delta$-subset $B$ of $G$ \cite{CR}. The following classical result is due to 
Comfort and Ross \cite{CR}:
\begin{theorem}
\label{CR:theorem}
A dense subgroup $D$ of a compact group $G$ is pseudocompact if and only if $D$ is $G_\delta$-dense in $G$.
\end{theorem}

The following question was asked by Hern\'andez, Macario and Trigos-Arrieta in 
\cite[Question 5.12(iii)]{HMT}:

\begin{question}\label{q3}
Does there exist (in ZFC) a non-metrizable compact abelian group $G$ such that every $G_\delta$-dense subgroup $D$ of $G$ determines $G$?
\end{question}

This question was also repeated in \cite[Question 4.12]{DFH}.

It is useful to state explicitly the negation of the statement in Question \ref{q3}:

\begin{question}\label{P-form}
Let $G$ be a compact abelian group such that every $G_\delta$-dense subgroup of $G$ determines $G$. Must $G$ be metrizable (in ZFC)?
\end{question}

By Theorem \ref{CR:theorem}, one can replace ``$G_\delta$-dense''
by ``dense pseudocompact'' in both questions to get their equivalent versions.

Theorem \ref{CHMRT} says that a compact abelian group $G$ is metrizable provided that every dense subgroup of $G$ determines it. Since
$G_\delta$-dense subgroups of $G$ are dense in $G$, a positive answer to Question \ref{P-form} (equivalently, a negative 
answer to Question \ref{q3}) would provide a strengthening of  Theorem \ref{CHMRT}, because one would get the same conclusion 
under a weaker assumption of requiring only a much smaller family of $G_\delta$-dense subgroups of $G$ to determine it.
One of the goals of this paper is to accomplish precisely this, without  recourse to any additional set-theoretic assumptions beyond Zermelo-Fraenkel axioms
ZFC of set theory.

\begin{remark}
\label{consistent:solution}
Chasco, Dom\'\i nguez and Trigos-Arrieta proved recently that every compact abelian group $G$ with  $w(G)\ge \cont$ has a $G_\delta$-dense subgroup which does not determine $G$ \cite[Theorem 14]{CDT}. Independently, Bruguera and Tkachenko proved  that every compact abelian group $G$ with $w(G)\ge \cont$ contains a proper $G_\delta$-dense reflexive subgroup $D$ \cite[Theorem 4.7]{BT}. 
As mentioned in the the end of \cite[Section 3]{CDT}, this $D$ cannot determine $G$.
(Indeed, $\widehat{\widehat{D}}=D\not=G=\widehat{\widehat{G}}$ implies $\widehat{D}\not=\widehat{G}$.)
It is clear that, 
{\em 
under the assumption of the Continuum Hypothesis, these results yield a consistent positive answer to Question \ref{P-form} and therefore, a consistent negative 
answer to Question \ref{q3}.
}
\end{remark}

An overview of the paper follows. Inspired by Questions~\ref{q3} and \ref{P-form}, in
Section~\ref{necess:section} we introduce four properties that every dense 
subgroup determining a compact abelian group must have (see Diagram 1),
thereby making a first attempt to clarify the ``fine structure'' of the notion of determination.
Section~\ref{sec:properties} collects some basic facts about the 
introduced properties that help the reader in better understanding of these new notions.

In Section \ref{metrization:section}, we investigate what happens to a compact  group when all its dense (or all its $G_\delta$-dense) subgroups are assumed to have one of the four properties introduced in Section~\ref{necess:section}. Our results in Section \ref{metrization:section} substantially clarify the ``fine structure'' of the notion of a determined group by addressing the following question systematically: ``How much determination'' of a compact group is really necessary in order to get its metrization in the spirit of Theorem \ref{CHMRT} or Question~\ref{P-form}? As it turns out, such metrization criteria can be obtained under much weaker conditions than full determination; see Theorems \ref{nonAV:inside} and \ref{ZFC}. In turn, Theorems \ref{strong:limit}, \ref{nonAV:Gdelta} and \ref{cc:are:sAV} serve to demonstrate that the conditions equivalent to the metrization of a compact group in Theorems \ref{nonAV:inside} and \ref{ZFC}  are the best possible, thereby pinpointing the exact property among the four necessary conditions ``responsible'' for both the validity of Theorem \ref{CHMRT} and the positive answer to Question~\ref{P-form}. The answer to Question~\ref{P-form} itself comes as a particular corollary of the main result; see Corollary \ref{ZFC:corollary}. An added bonus of our approach is that many results in this section 
hold for non-abelian compact groups as well, whereas the notion of a determined group is restricted to the abelian case. 
(A non-commutative version of a determined group was introduced recently in \cite{FH}.)

In Section \ref{killing:uncountable:compacta} we develop a  machinery 
for constructing  $G_\delta$-dense subgroups $D$ without uncountable compact subsets in compact groups $G$ 
of weight $\omega_1$.
Furthermore, when $G$ belongs to a fixed variery $\V$ of groups, the subgroup $D$ can be chosen to be a free group in the variety $\V$. 
Our machinery works in ZFC alone.
As in 
Section \ref{metrization:section}, results in this section do not 
require
$G$ to be abelian. 
The primary novelty here is our ability to handle successfully small weights of $G$ (like $\omega_1$) at the expense of ``killing'' only {\em uncountable\/} compact subsets of $D$.
All known constructions in the literature usually ``kill'' all {\em infinite\/} compact subsets, thereby eliminating also all non-trivial convergent sequences in $D$, but this stronger conclusion is accomplished at the expense of having been able to handle only groups $G$ of weight $\cont$. In fact, this difference is inherent in the nature of the problem and not purely coincidental. Indeed, Remark \ref{MA:remark} shows that the group $D$ we construct must have non-trivial convergent sequences under some additional set-theoretic assumptions.
As a particular corollary of our results, we produce a pseudocompact group topology on the free group $F_{\cont}$ with $\cont$ many generators without uncountable compact subsets. A recent result by Thom \cite{Thom} implies that such 
a topology on $F_{\cont}$ must necessarily contain a non-trivial convergent sequence; see Remark \ref{free:remark}.

Sections \ref{basic:section}, \ref{proof:of:variety:theorem} and \ref{sec:8}
are devoted to the proofs of 
the main results from Sections \ref{metrization:section} and \ref{killing:uncountable:compacta}. Section \ref{sec:examples} contains some examples showing
the 
limits of our results, and Section \ref{sec:problems} lists open problems related to the topic of this paper.

\section{Four necessary conditions for determination of a compact abelian group}
\label{necess:section}

In this section we introduce four conditions and show that they are all necessary for determination of a compact abelian group.

\begin{definition}
\label{def:w-compact:AV}
Let $X$ be a space.
\begin{itemize}
\item[(i)] We shall say that $X$ is {\em $w$-compact\/} if there exists a compact subset $C$ of $X$ such that $w(C)=w(X)$.
\item[(ii)] We shall say that $X$ has
the {\em \AVp\/} (or is an {\em \AV\ space\/}) provided that $w(X)\le |X|$.
\end{itemize} 
\end{definition}
The letter $w$ in front of ``compact'' in item (i) is intended to abbreviate the word ``weight'', but one can also view it as an abbreviation of the word 
``weak'', as every compact space is obviously $w$-compact. 

The name for the class of spaces in item (ii) was chosen to pay tribute to the first manuscript of Professor Arhangel'ski\u{\i} \cite{Arh0} where he introduced the notion of 
network weight and demonstrated its importance in the study of compact spaces.
A celebrated result of 
Arhangel'ski\u\i\ from  \cite{Arh0} says that $w(X)=nw(X)\le|X|$ for every 
compact space $X$. In our terminology, this means 
that every compact space has the 
 \AVp. 
In fact, a bit more can be said. Indeed, let $X$ be a $w$-compact space. Then $X$ contains a compact subset $C$ such that $w(C)=w(X)$. Combining this with the above result
of Arhangel'ski\u\i, we obtain $w(X)=w(C)\le |C|\le |X|$. Therefore, $X$ has the \AVp.
This argument shows that
\begin{itemize}
\item[($\alpha$)]
 a $w$-compact space has the \AVp.
\end{itemize}

\begin{definition}
\label{def:projective}
Let $G$ be a topological group. 
\begin{itemize}
\item[(i)] We shall say that $G$ is {\em projectively $w$-compact\/} if every continuous homomorphic image of $G$ is $w$-compact.
\item[(ii)]  We shall say that $G$ is {\em \sAV\/} if every continuous homomorphic image of $G$  has the \AVp. 
\end{itemize}
\end{definition}

Since compactness is preserved by continuous images and compact spaces are $w$-compact, all compact groups are projectively $w$-compact.
 From ($\alpha$) and Definition \ref{def:projective}(ii) we get
\begin{itemize}
\item[($\beta$)]
projectively $w$-compact groups are \sAV.
\end{itemize}

The following necessary condition for determination was found by the authors in \cite{DS2}. Since it plays a crucial role in the present paper,
we provide a shorter self-contained proof of this result requiring no recourse to the  notion of qc-density that was essential in \cite{DS2}.

\begin{theorem}
\label{size:of:determined:subgroup} \cite[Corollary 2.4]{DS2}
If a subgroup $D$ of an infinite compact abelian group $G$ determines $G$, then $D$ contains a compact subset $X$ such that $w(X)=w(D)$.
\end{theorem}
\begin{proof}
For a subset $X$ of $G$ and an open neighbourhood $V$ of $0$ in $\T$, let
$W(X,V)=\{\chi\in\widehat{G}: \chi(X)\subseteq V\}$.
Since $D$ determines $G$ and $\widehat{G}$ is discrete, there exists a compact subset $X$ of $D$
and an open neighbourhood $V$ of $0$ such that $W(X,V)=\{0\}$.
Let $\pi:\widehat{G}\to C(X,\T)$ be the restriction homomorphism defined by $\pi(\chi)=\chi\restriction_X$ for $\chi\in\widehat{G}$,
where $C(X,\T)$ denotes the group of all continuous functions from $X$ to $\T$ equipped with 
the compact-open topology.
Since $\ker \pi\subseteq W(X,V)=\{0\}$, $\pi$ is a monomorphism, and so 
$w(G)=|\widehat{G}|=|H|$, where $H=\pi(\widehat{G})$.
Furthermore, 
$U\cap H=\{0\}$, where $U=\{f\in C(X,\T): f(X)\subseteq V\}$ is an open subset of $C(X,\T)$, so $H$ is a discrete subgroup of $C(X,\T)$.
Therefore, $|H|=w(H)\le w(C(X,\T))=w(X)+\omega$ by \cite[Proposition 3.4.16]{E}. This proves that $w(G)\le w(X)+\omega$.
To finish the proof of the inequality $w(G)\le w(X)$, it suffices to show that $X$ is infinite. Indeed, assume that  $X$ is finite. Then $C(X,\T)=\T^X$ is compact, 
and so the discrete subgroup $H$ of $C(X,\T)$ must be finite. This contradicts the fact that $|H|=|\widehat{G}|\ge\omega$, as $G$ is infinite. Finally, the reverse inequality $w(X)\le w(G)$ is clear.
\end{proof}

The relevance of the four notions introduced in Definitions 
\ref{def:w-compact:AV}
and
\ref{def:projective}
to the topic of our paper is evident from the following corollary of this theorem.
\begin{corollary}
\label{tot:ps:min}
If a subgroup $D$ of a compact abelian group determines it, then $D$ is projectively $w$-compact.
\end{corollary}

\begin{proof}
Let $D$ be a dense subgroup of a compact abelian group $G$ that determines $G$, and let $f:D\to N$ be a continuous homomorphism 
onto some topological group $N$. Then $f$ can be extended to a continuous group homomorphism from $G$ to the completion $H=\widehat{N}$ of $N$,
and we denote this extension by the same letter $f$. Since $D$ determines $G$, the dense subgroup $f(D)$ 
of the compact group $f(G)=H$ determines $H$ \cite[Corollary 3.15]{CRT1}.
If $H$ is finite, then $f(D)=H$ is compact, so trivially $w$-compact. 
If $H$ is infinite, we apply Theorem \ref{size:of:determined:subgroup} to conclude that 
$f(D)$ contains a compact set $X$ with $w(X)=w(H)=w(f(D))$. That is, $f(D)$ is $w$-compact. 
This shows that $D$ is projectively $w$-compact.
\end{proof}

The relations between the properties introduced above in the class of precompact abelian groups can be summarized 
in the following diagram: 

\medskip
\begin{center}
$\phantom{MM}
{\xymatrix@!0@C5.4cm@R=1.5cm{
 \mbox{compact} \ar@{->}[r]  &  \mbox{determining the completion}\ar@{->}[d]^{(\ref{tot:ps:min})}   & \mbox{metrizable}\ar@{->}[l]_{\hskip50pt(\ref{CA})}  \\
  & \mbox{projectively $w$-compact}\ar@{->}[d] \ar@{->}[r]^{\!\!(\beta)}  & \mbox{\sAV}\ar@{->}[d] \\
   &\mbox{$w$-compact} 
\ar@{->}[r]^{(\alpha)} &  \mbox{\AV} \\
    } }$
\end{center}
\smallskip
\begin{center}
Diagram 1.
\end{center}

\medskip

This diagram shows that four properties from Definitions \ref{def:w-compact:AV} and \ref{def:projective}
are necessary for determination of the completion of a precompact abelian group.
With an exception of the arrow (\ref{tot:ps:min}), none of the other arrows in Diagram 1 are invertible.

A dense subgroup $D$ of a compact group $G$ that determines $G$ 
need not
be either compact or metrizable. To see this, it suffices to recall that the direct sum 
$\bigoplus_{\alpha<\omega_1} \T$ of $\omega_1$ copies of $\T$ 
determines $\T^{\omega_1}$; see \cite[Corollary 3.12]{CRT1}. 

In Example \ref{proj:AV:not:proj:W-compact},
we exhibit a pseudocompact \sAV\ group $D$ that is not $w$-compact. (Furthermore, under the assumption of the CH, 
$D$ can be chosen to be even countably compact.) In particular, neither the arrow ($\alpha$) nor the arrow ($\beta$) is reversible.

For every infinite cardinal $\kappa$, there exists a $\kappa$-bounded $w$-compact (thus, \AV) abelian group that is not \sAV\ (and so 
is not projectively $w$-compact); see Example \ref{ex:omega-bounded}.

We do not know if the arrow (\ref{tot:ps:min}) in Diagram 1 is invertible. In fact,
it is tempting to conjecture that Corollary \ref{tot:ps:min} gives not only a necessary but also a sufficient condition for determination of a compact abelian group by its dense subgroup.

\begin{question}
\label{question:w-compact}
Does every dense projectively $w$-compact subgroup of a compact abelian group determine it?
\end{question}

We refer the reader to Remark \ref{totally:minimal:remark}(ii) for a partial positive answer to this question.

\section{Properties of \AV\ spaces and \sAV\ groups}
\label{sec:properties}

Our first remark shows that the \AVp\ is ``local''.

\begin{remark}
\label{AV:remark} 
{\em For every space $X$,  the inequalities $w(X)\le |X|$  and $\chi(X)\le |X|$ are equivalent.\/} 
\end{remark}

\begin{proposition}
\label{AV:proposition}
\begin{itemize}
  \item[(i)] Locally compact spaces have the \AVp.
  \item[(ii)] First countable (in particular, metric) spaces have the \AVp.
  \item[(iii)] The class of \AV\ spaces is closed under taking perfect preimages; that is, if $f:X\to Y$ is a perfect map from a space $X$ onto an \AV\ space $Y$, then $X$ has the \AVp.
  \item[(iv)]
If $w(X)$ is a strong limit cardinal, then $X$ has the \AVp.
\end{itemize}
\end{proposition}
\begin{proof}
(i) Let $X$ be a locally compact space. If $X$ is finite, then $X$ has the \AVp. Suppose that $X$ is infinite. Since the one-point compactification  $Y$ of $X$ is compact, it has the \AVp, so $w(Y)\le |Y|$. 
Since $X$ is infinite and $Y\setminus X$ is a singleton, $|Y|=|X|$. Since $X$ is a subspace of $Y$, we get $w(X)\le w(Y)$. This proves that  $w(X)\le |X|$.

(ii) For finite spaces $X$, this follows from (i). If $X$ is infinite,  then the conclusion follows from Remark \ref{AV:remark}.

(iii) 
Since finite spaces have the \AVp\ by (i), we shall assume that $X$ is infinite. There exists a one-to-one continuous map $g:X\to Z$ onto a space $Z$ such that $w(Z)\le nw(X)\le |X|$ \cite{Arh0}. Let $h:X\to Y\times Z$ be the diagonal product 
of $f$ and $g$ defined by $h(x)=(f(x),g(x))$ for all $x\in X$. Since $f$ is a perfect  map, so is $h$ \cite[Theorem 3.7.9]{E}.
Since $g$ is one-to-one, $h$ is an injection. It follows that $X$ and $h(X)$ are homeomorphic, so 
\begin{equation}
\label{eq.1}
w(X)\le w(h(X))\le w(Y\times Z)=\max\{w(Y),w(Z)\}\le \max\{w(Y),|X|\}.
\end{equation}
Since $Y$ has the \AVp, $w(Y)\le |Y|=|f(X)|\le |X|$. Combining this with \eqref{eq.1}, we conclude that  $w(X)\le |X|$. Thus, $X$ has the \AVp.

(iv)
Since $d(X)\le w(X)$, $w(X)\le 2^{d(X)}$ and $w(X)$ is a strong limit cardinal,  $w(X)=d(X)\le|X|$.
\end{proof}

\begin{proposition}
If a topological group $G$ contains a dense subgroup $H$ with the \AVp, then $G$ itself  has the \AVp.
\end{proposition}

\begin{proof}
Since $H$ is dense in $G$, $\chi(H)=\chi(G)$. Since $H$ has the \AVp, $\chi(H)\le w(H)\le |H|$. Since $H$ is a subgroup of $G$, $|H|\le |G|$. This shows that $\chi(G)\le |G|$. Therefore, $G$ has the \AVp\ by Remark \ref{AV:remark}.
\end{proof}

This proposition does not hold for spaces since one may have $w(Y)<w(X)$ when $Y$ is  a dense subspace of $X$.

\begin{proposition}
\label{small:psc:are:sAV}
Every pseudocompact group $G$ such that $w(G)\le\cont$ is \sAV.
\end{proposition}

\begin{proof}
Indeed, let $f:G\to H$ be a continuous surjective homomorphism of $G$ onto a topological group $H$.  Then $H$ is pseudocompact, as  a continuous image of the pseudocompact space $G$. If $H$ is finite, then $H$ has the \AVp\ by Proposition \ref{AV:proposition}(i). Assume now that $H$ is infinite. Then $|H|\ge\cont$ \cite[Proposition 1.3(a)]{vD}. To show that $H$ has the \AVp, it suffices to note that $w(H)\le\cont$. Indeed, let $\widehat{f}:\widehat{G}\to \widehat{H}$ be the extension of $f$ over the completion $\widehat{G}$ of $G$. Since $\widehat{G}$ is compact and $\widehat{f}$ is surjective,  $w(H)=w(\widehat{H})\le w(\widehat{G})=w(G)\le\cont$.
\end{proof}

Item (i) of our next proposition shows that the restriction on weight  in Proposition \ref{small:psc:are:sAV} is the best possible, 
while
item (ii) of 
Proposition \ref{ex1}
shows that even groups ``arbitrarily close'' to compact  need not have 
the \AVp.
(Compare this with Proposition~\ref{AV:proposition}(i).)

\begin{proposition}
\label{ex1}
\begin{itemize}
\item[(i)]
Every compact group $G$ with $w(G)=\cont^+$ has  a dense countably compact subgroup  without the \AVp.
\item[(ii)]
For every infinite cardinal $\kappa$, each compact group $G$ of weight $\tau=2^{2^{2^\kappa}}$
has a dense $\kappa$-bounded subgroup 
without the
\AVp.
\end{itemize}
\end{proposition}

\begin{proof}
(i) Since $\cont^+\le 2^\cont$, applying 
\cite[Theorem 2.7]{I} we can choose a dense subgroup $H$ of $G$
such that $|H|=\cont$. By the standard closing-off argument, we can find a countably compact subgroup $D$ of $G$ such that $H\subseteq D$ and $|D|\le \cont$.   Since $H$ is dense in $G$, so is $D$. Since $|D|=\cont<\cont^+=w(G)=w(D)$,  $D$ does not have the \AVp.

(ii) 
By 
\cite[Theorem 2.7]{I},
$G$ contains a dense subgroup $H$ of size $2^{2^\kappa}$.  Let $D$ be the $\kappa$-closure of $H$ in $G$; that is, $D = \bigcup \left\{\overline{A}: A \in [H]^{\leq \kappa }\right\}$, where  $\overline{A}$ denotes the closure of $A$ in $G$. Clearly, $D$ is a subgroup of $G$ containing $H$, so $D$ is dense in $G$. Since $\left|[H]^{\leq \kappa} \right|\leq 2^{2^\kappa}$ and $|\overline{A}|\leq 2^{2^\kappa}$ for every $A \in [H]^{\leq \kappa}$, we conclude that $|D|\le 2^{2^\kappa} <2^{2^{2^\kappa}}=w(D)$. Therefore, $D$ does not have the \AVp.  
\end{proof}

\section{Metrizability of compact groups via conditions on their dense subgroups}
\label{metrization:section}

Our first theorem demonstrates that the weakest condition in Diagram 1 is not sufficient for getting the metrizability of a compact group $G$ even when this condition is imposed 
on all dense subgroups of $G$.

\begin{theorem}
\label{strong:limit}
Every dense subgroup of a compact  group $G$ has the \AVp\  if and only if $w(G)$ is a strong limit cardinal.
\end{theorem}

Our second theorem shows that the projective version of the weakest condition in Diagram 1 imposed on {\em all\/} dense subgroups of a compact group $G$  suffices to obtain its metrizability. 

\begin{theorem}
\label{nonAV:inside}
Every dense subgroup of a compact group $G$ is \sAV\ if and only if $G$ is metrizable.
\end{theorem}

Since a dense determining subgroup of a compact abelian group is \sAV\  (see Diagram~1),  in the abelian case 
the ``only if'' part of
this result strengthens Theorem \ref{CHMRT} by offering the same conclusion under a much weaker assumption.

For a cardinal $\sigma$, the minimum cardinality of a pseudocompact group of weight $\sigma$ is denoted by $m(\sigma)$  \cite{CoRo}.

The next theorem is a counterpart of Theorem \ref{strong:limit}  for $G_\delta$-dense subgroups.
 
\begin{theorem}\label{nonAV:Gdelta}
Every $G_\delta$-dense subgroup of a compact group $G$  has the \AVp\ if and only if $m(w(G))\ge w(G)$. 
\end{theorem}
 
Our next result is  the counterpart of Theorem \ref{nonAV:inside} with ``dense'' replaced by ``$G_\delta$-dense''.

\begin{theorem}
\label{cc:are:sAV}
For a compact group $G$, the following conditions are equivalent:
\begin{itemize}
  \item[(i)] every $G_\delta$-dense (equivalently, each dense pseudocompact) subgroup of $G$ is \sAV;
  \item[(ii)] all dense countably compact subgroups of $G$ are \sAV;
  \item[(iii)] $w(G)\le\cont$.
\end{itemize}
\end{theorem}

This theorem shows that having all $G_\delta$-dense subgroups of 
a compact group 
$G$ \sAV\ is not sufficient for obtaining metrizability of 
$G$. Our next theorem shows
that strengthening ``\sAV'' to ``projectively $w$-compact'' yields metrizability of $G$ in case when $G$ is either connected or abelian. 

\begin{theorem}\label{ZFC}
Let $G$ be a compact group that is either abelian or connected.
If all $G_\delta$-dense (equivalently, all dense pseudocompact) subgroups of $G$ are projectively $w$-compact, then $G$ is metrizable. 
\end{theorem}

Combining this result with Corollary \ref{tot:ps:min}, we obtain the following corollary solving  Question \ref{q3} in the negative and Question \ref{P-form} in the positive. 

\begin{corollary}
\label{ZFC:corollary}
If all $G_\delta$-dense subgroups of a compact abelian group  $G$ determine it, then $G$ is metrizable.
\end{corollary}

Under the assumption of the Continuum Hypothesis, the following stronger version of Theorem \ref{ZFC} can be obtained in the abelian case.

\begin{theorem}\label{CH}
Assume CH. If all dense countably compact subgroups of a compact abelian group $G$ are projectively $w$-compact, then $G$ is metrizable.
\end{theorem}

Since countable compactness is stronger then pseudocompactness and a dense pseudocompact subgroup of a compact abelian group  is $G_\delta$-dense in it (Theorem \ref{CR:theorem}), our Theorem \ref{CH} strengthens also the consistent result typeset in italics in Remark \ref{consistent:solution}. 

The proofs of Theorems \ref{strong:limit},  \ref{nonAV:inside}, \ref{nonAV:Gdelta}, \ref{cc:are:sAV} are postponed until Section \ref{basic:section}, 
while the
proofs of Theorems \ref{ZFC} and \ref{CH} are postponed until Section \ref{sec:8}.

Let $G$ be any compact abelian group of weight $\omega_1$. It follows from Theorem \ref{cc:are:sAV} that
all $G_\delta$-dense subgroups of $G$ are \sAV, even though $G$ is not metrizable. This shows that ``projectively $w$-compact'' cannot be weakened to 
``\sAV'' in the assumption of Theorems \ref{ZFC} and \ref{CH}.
Furthermore, 
since $\omega_1$ is  not a strong limit cardinal, Theorem \ref{strong:limit} implies that $G$ has a dense subgroup without the \AVp.
Combining Theorem  \ref{strong:limit}  with 
Example~\ref{psc:example}(ii) below, 
we obtain
compact abelian groups $G$  of arbitrarily large weight
such that  every $G_\delta$-dense subgroup of $G$ has the \AVp, but there exists a dense subgroup of $G$ without the \AVp. 

We finish this section with the following corollary of its main results.

\begin{corollary}
\label{huge:corollary}
For a compact abelian group $G$, the following conditions are equivalent:
\begin{itemize}
  \item[(i)] $G$ is metrizable;
  \item[(ii)] every dense subgroup of $G$ determines $G$;
  \item[(iii)] every $G_\delta$-dense (equivalently, each dense pseudocompact)  subgroup of $G$ determines $G$;
  \item[(iv)]  every dense subgroup of $G$ is \sAV;
  \item[(v)] every  $G_\delta$-dense (equivalently, each dense pseudocompact)  subgroup of $G$ is projectively $w$-compact. 
\end{itemize}
Furthermore, under CH, the following two items can be added to the list of equivalent conditions (i)--(v):
\begin{itemize}
  \item[(vi)] every dense countably compact subgroup of $G$ determines $G$;
  \item[(vii)] every dense countably compact subgroup of $G$ is projectively $w$-compact.
\end{itemize}
\end{corollary}

\begin{proof}
(i)$\to$(ii) is Theorem \ref{CA}, (ii)$\to$(iv) follows from Diagram 1, (iv)$\to$(i) 
follows from
Theorem \ref{nonAV:inside}.

(i)$\to$(iii) follows from Theorem \ref{CA}, (iii)$\to$(v) follows from 
Corollary \ref{tot:ps:min}, (v)$\to$(i) is Theorem \ref{ZFC}.

(i)$\to$(vi) follows from Theorem \ref{CA}, (vi)$\to$(vii) follows from  Corollary \ref{tot:ps:min}. Finally, (vii)$\to$(i) is Theorem \ref{CH}. (We note that only the last implication needs CH.)
\end{proof}

\section{Pseudocompact groups of small weight without uncountable compact subsets}
\label{killing:uncountable:compacta}

For a subset $X$ of a group $G$ we denote by $\grp{X}$ the subgroup of $G$ generated by $X$.

By a {\em variety of groups\/} we mean, as usual, a class of groups closed under taking Cartesian products, subgroups and quotients (i.e., a {\em closed class\/} in the sense of Birkhoff \cite{Bi}). Another, equivalent, way of defining a variety is by giving a  fixed family of identities satisfied by all groups of the variety (\cite{Bi}; see also \cite[Theorem 15.51]{N}).

\begin{definition}
\label{variety:def}
Let $\V$ be a variety of groups. 
\begin{itemize}
\item[(a)]
Recall that a subset $X$ of a group $G$ is called {\em $\V$-independent\/} provided that the following two conditions are satisfied: 
\begin{itemize}
  \item[(i)] $\grp{X}\in \V$;
  \item[(ii)] for every map $f:X\to G$ with $G\in\V$, there exists a homomorphism $\tilde{f}:\grp{X}\to G$ extending $f$.
\end{itemize}
  \item[(b)] For every group $G\in\V$ the cardinal $r_\V(G)=\sup\{|X|: X$ is $\V$-independent subset of $G\}$
is called the {\em $\V$-rank\/} of $G$.
  \item[(c)] A group $G$ is {\em $\V$-free\/} if $G$ is generated by its $\V$-independent subset $X$. We call this $X$ the {\em generating set\/}  (or the set of {\em generators\/} of $G$ and we write $G=F_\V(X)$.
\end{itemize}
\end{definition}

\begin{theorem}
\label{variety:theorem}
Let $\V$ be a variety of groups and $L$ be a compact metric group that belongs to $\V$ such that $r_\V(L^\omega)\ge\omega$. Let $I$ be a set such that $\omega_1\le|I|\le\cont$. Then the group $L^I$ contains a $G_\delta$-dense (so dense pseudocompact) $\V$-free subgroup $D$ of cardinality $\cont$  such that all compact subsets of $D$ are countable; in particular $D$ is not w-compact.
\end{theorem}

The proof of this theorem is postponed until Section \ref{proof:of:variety:theorem}.

\begin{corollary}
\label{Lie:corollary}
Let $L$ be a compact simple Lie group. Then for every uncountable set $I$ of size at most $\cont$, the group $L^I$ contains a $G_\delta$-dense free subgroup $D$ of cardinality $\cont$ such that all compact subsets of $D$ are countable;  in particular $D$ is not w-compact.
\end{corollary}

\begin{proof}
By \cite[Theorem 2]{BM}, $r_{\mathcal{G}}(L^\omega)\ge r_{\mathcal{G}}(L)\ge\omega$, where $\mathcal{G}$ is the variety of all groups. Now we can  apply Theorem \ref{variety:theorem} with $\V=\mathcal{G}$.
\end{proof}

\begin{corollary}
\label{semi-Bernstein:theorem}
For every non-trivial compact metric abelian group $L$ and every uncountable set $I$ of size at most $\cont$, the group $L^I$ contains a $G_\delta$-dense subgroup $D$ of cardinality $\cont$ such that all compact subsets of $D$ are countable; in particular $D$ is not w-compact. Furthermore, if $L$ is unbounded, then $D$ can be chosen to  be free.
\end{corollary}
\begin{proof}
We consider two cases.

\smallskip
{\em Case 1\/}. {\em $L$ is bounded\/}. Let $n$ be the order of $L$, and let $\mathcal{A}_n$ be the variety of abelian groups of order $n$.
Then $L\in\mathcal{A}_n$ and $r_{\mathcal{A}_n}(L^\omega)\ge\omega$,  so the conclusion follows from Theorem \ref{variety:theorem} applied  to $\V=\mathcal{A}_n$.

\smallskip
{\em Case 2\/}. {\em $L$ is unbounded\/}. Let $\mathcal{A}$ be the variety of all  abelian groups. Then $r_{\mathcal{A}}(L^\omega)\ge\omega$, so the conclusion follows from Theorem \ref{variety:theorem} applied to $\V=\mathcal{A}$.
\end{proof}

Following \cite[Definition 5.2]{DS-Mem}, we say that a variety $\V$ is {\em precompact\/} if $\V$ is generated by its finite groups. One can find a host of conditions equivalent to precompactness of a variety in  \cite[Lemma 5.1]{DS-Mem}. 
In particular,  it is worth noting in connection with Theorem \ref{variety:theorem} that the existence of a compact group $L \in \V$ with $r_\V(L)\ge\omega$ is equivalent to precompactness of the variety $\V$ \cite[Lemma 5.1]{DS-Mem}.

Most of the well-known varieties are precompact; see \cite[Lemma 5.3]{DS-Mem} and the comment following this lemma. The Burnside variety $\mathcal{B}_n$ for odd $n>665$ is not precompact \cite{D}. 

\begin{corollary}
\label{V-free:group:corollary}
For a variety $\V$, the following conditions are equivalent:
\begin{itemize}
\item[(i)] $\V$ is precompact;
\item[(ii)]
for every cardinal $\sigma$ with $\omega_1\le\sigma\le\cont$, the $\V$-free group with $\cont$ many generators admits a  pseudocompact group topology of weight $\sigma$ without uncountable compact subsets;
in particular, this topology is not $w$-compact. 
\end{itemize}
\end{corollary}

\begin{proof}
(i)$\to$(ii)
Suppose that $\V$ is precompact. By \cite[Lemma 5.1]{DS-Mem}, there exists  a compact metric group $L\in\V$ with $r_\V(L)\ge\omega$. Since  $r_\V(L^\omega)\ge r_\V(L)$, applying Theorem \ref{variety:theorem} we get (ii).

(ii)$\to$(i) This follows from \cite[Theorem 5.5]{DS-Mem}.
\end{proof}

Our next remark shows that Theorem \ref{variety:theorem} and its Corollaries \ref{Lie:corollary}, \ref{semi-Bernstein:theorem} and \ref{V-free:group:corollary} are the best possible results  that one can obtain in ZFC.

\begin{remark}
\label{MA:remark}
Assume MA+$\neg$CH, where MA stands for Martin's Axiom. 
In Theorem \ref{variety:theorem} and Corollaries \ref{Lie:corollary}, \ref{semi-Bernstein:theorem},
take $I$ to be a set of size $\omega_1$, 
and let $D$ be the group as in 
the conclusion of these results.
In Corollary \ref{V-free:group:corollary}, let $\sigma=\omega_1$ and let $D$ denote the $\V$-free group with $\cont$ many generators.
Then $D$ is a topological group of weight $\omega_1<\cont$. Since MA holds, every countable subgroup of $D$ is Fr\'{e}chet-Urysohn \cite{MS}; 
in particular, $D$ contains many non-trivial convergent sequences. Therefore, 
``all compact subsets of $D$ are countable'' cannot be 
strengthened to ``all compact subsets of $D$ are finite''
in conclusions of 
Theorem \ref{variety:theorem} and Corollaries \ref{Lie:corollary}, \ref{semi-Bernstein:theorem},
and 
``without uncountable compact subsests''
cannot be 
strengthened to 
``without infinite compact subsests''
in the the conclusion of
 Corollary \ref{V-free:group:corollary}.
\end{remark}

Recall that the strongest totally bounded group topology on a group is called  its {\em Bohr topology\/}.

\begin{remark}
\label{free:remark}
Thom recently proved that the free group with two generators equipped with its Bohr topology contains a non-trivial convergent sequence \cite{Thom}. This easily implies that {\em every\/} precompact group topology on the free group with two generators contains a non-trivial convergent sequence. Since pseudocompact groups are precompact, it follows that {\em every pseudocompact  free group of size $\cont$ contains a non-trivial convergent sequence\/}.
Combining this with Theorem \ref{CR:theorem}, we conclude that the group $D$  as in the conclusion of Corollary \ref{Lie:corollary} contains a non-trivial convergent sequence.
This shows that 
``all compact subsets of $D$ are countable'' cannot be strengthened to
``all compact subsets of $D$ are finite''
in the conclusion of 
Corollary \ref{Lie:corollary}  and 
``without uncountable compact subsests''
cannot be 
strengthened to 
``without infinite compact subsests''
in the the conclusion of
Corollary \ref{V-free:group:corollary} when $\V$ is the variety of all  groups.
\end{remark}

\section{Proofs of Theorems \ref{strong:limit},  \ref{nonAV:inside}, \ref{nonAV:Gdelta}, \ref{cc:are:sAV}.}
\label{basic:section}

\medskip
\noindent
{\bf Proof of Theorem \ref{strong:limit}:}
Suppose that $w(G)$ is not a strong limit cardinal. Then  there exists  a dense subgroup $D$ of $G$ such that $|D|=d(G)<w(G)=w(D)$;
see \cite[Theorem 2.7]{I}. Hence, $D$ does not have the \AVp.

Suppose now that $w(G)$ is a strong limit cardinal. Let $D$ be a dense subgroup of $G$. Since $w(D)=w(G)$, the cardinal $w(D)$ is strong limit. Hence, $D$ has the \AVp\ by Proposition \ref{AV:proposition}~(iv).
\qed

\medskip

 \noindent
 {\bf Proof of Theorem \ref{nonAV:Gdelta}:}
Assume that every $G_\delta$-dense subgroup of $G$ has the \AVp. According to \cite{CoRo}, $G$ has a $G_\delta$-dense subgroup $D$ of size $m(\sigma)$.  Since $D$ has the \AVp, this yields $m(\sigma)=|D|\ge w(D) = w(G) = \sigma$. Conversely, if $m(\sigma)\ge\sigma$ holds, then for every $G_\delta$-dense subgroup $D$ of $G$, one has  $|D|\geq m(\sigma) \geq \sigma = w(D)$,  so $D$ has the \AVp. 
\qed

\begin{fact}
\label{cont:images:of:smaller:weight}
\cite[Lemma 1.5]{IS}
Let $G$ be an infinite compact group. For every infinite cardinal $\tau\le w(G)$ there exists 
a continuous homomorphism $f:G\to H$ of $G$ onto a compact group $H$ with $w(H)=\tau$. 
\end{fact}

\begin{fact}
\label{pull:back}
Suppose that $f:G\to H$ is a continuous surjective homomorphism of compact abelian groups, $D$ is a subgroup of $H$ and $D_1=f^{-1}(D)$. 
\begin{itemize}
\item[(i)] If $D$ is dense in $H$, then $D_1$ is dense in $G$.
\item[(ii)] If $D$ is pseudocompact (countably compact, $\kappa$-bounded for some infinite cardinal $\kappa$), then $D_1$ has the same property.
\item[(iii)] If $D$ is not (projectively) $w$-compact, then $D_1$ is not projectively $w$-compact either.
\item[(iv)] If $D$ is not (projectively) \AV, then $D_1$ is not \sAV\ either.
\end{itemize}
\end{fact}
\begin{proof}
(i) Let $L$ be the closure of the subgroup $D_1$ in $G$. Since $L\supseteq D_1 \supseteq \ker f$, $f(L)$ is a closed subgroup of $H$. Since it contains the dense subgroup $D$, we deduce that $f(L)=H$. Using again $L\supseteq \ker f$, we deduce that $L = G$, i.e., $D_1$ is dense in $G$.

(ii) Since the map $f$ is perfect, the conclusion follows from the well-known fact  that  the properties listed in item (ii) are preserved by taking  full preimages under perfect maps.  

(iii) and (iv) are straightforward.
\end{proof}

\medskip
\noindent
{\bf Proof of Theorem \ref{nonAV:inside}:}
The ``if'' part follows from Theorem \ref{CA} and Diagram 1.
Let us prove the ``only if'' part.
Let $G$ be a non-metrizable compact  group. By Fact \ref{cont:images:of:smaller:weight}, there exists a continuous group homomorphism $f:G\to H$ onto a compact 
group $H$ such that $w(H)=\omega_1$. Since $\omega_1$ is not a strong limit cardinal, we can use Theorem \ref{strong:limit}
to find a dense subgroup $D$ of $H$ without the \AVp. By Fact \ref{pull:back}, $D_1=f^{-1}(D)$ is a dense subgroup of $G$ that is not \sAV.
\qed
\medskip

\medskip
\noindent
{\bf Proof of Theorem \ref{cc:are:sAV}:}
(i)~$\to$~(ii) This implication is trivial, as all countably compact groups are pseudocompact.

(ii)~$\to$~(iii)
Let $G$ be a compact abelian group such that $w(G)\ge \cont^+$.  By Fact \ref{cont:images:of:smaller:weight}, there exists a continuous surjective homomorphism 
$f:G\to H$ onto a compact group $H$ such that $w(H)=\cont^+$.  By Proposition \ref{ex1}(i), $H$ has a dense countably compact subgroup $D$ 
without the \AVp. By Fact \ref{pull:back}, $D_1=f^{-1}(D)$ is a dense countably compact subgroup of $G$ that is not \sAV. This contradicts (ii).

(iii)~$\to$~(i) Indeed, let $D$ be a $G_\delta$-dense subgroup of $G$. Then $D$ is pseudocompact. Since $w(D)=w(G)\le\cont$, from Proposition \ref{small:psc:are:sAV} we conclude that $D$ is \sAV.  
\qed

\section{Proof of Theorem \ref{variety:theorem}}
\label{proof:of:variety:theorem}

\begin{lemma}
\label{unique:support}
Let $X$ be a set. For every $g\in F_\V(X)\setminus\{e\}$ there exists the unique non-empty finite set $F\subseteq X$ such that $g\in \grp{F}$ and $g\not\in \grp{F'}$ for every proper subset $F'$ of $F$. 
\end{lemma}
\begin{proof}
The existence of such an $F$ is clear. Suppose that $F_0$ and $F_1$ are finite subsets of $X$ such that $g\in \grp{F_i}$ and $g\not\in \grp{F'_i}$ for every proper subset $F'_i$ of $F_i$ ($i=0,1$).  Let $F'=F_0\cap F_1$, so that $F'\subseteq F_i$ for $i=0,1$.

Fix $i=0,1$. Let $f:X\to F_\V(X)$ be the map that coincides with the identity on $F_i$ and sends every element $x\in X\setminus F_i$ to $e\in  F_\V(X)$.
Since $X$ is $\V$-independent, $F_\V(X)=\grp{X}\in\V$ by item (i) of  Definition \ref{variety:def}(a), so we can use item (ii) of the same definition
to find a homomorphism $\tilde{f}:F_\V(X)\to F_\V(X)$ extending $f$. Since $g\in\grp{F_i}$ and $f$ is the identity on $F_i$, we conclude that $\tilde{f}(g)=g$. Since $g\in \grp{F_{1-i}}$, we have 
\begin{equation}
\label{lin:comb}
g=\tilde{f}(g)\in \grp{f(F_{1-i})}=\grp{f(F_{1-i}\cap F_i)\cup f(F_{1-i}\setminus F_i)}
=
\grp{f(F')\cup\{e\}}=\grp{f(F')}
=
\grp{F'}.
\end{equation}
Since $F'\subseteq F_i$, from $f\in \grp{F_i}$, \eqref{lin:comb} and our assumption on $F_i$ we conclude that $F_i=F'=F_0\cap F_1=F_i\cap F_{1-i}$. This proves that $F_i\subseteq F_{1-i}$.

Since the last inclusion holds for both $i=0,1$, it follows that $F_0=F_1$, as required.
\end{proof}

For every $g\in F_\V(X)\setminus\{e\}$ we denote by $\supp{g}{X}$ the unique set $F\subseteq X$ as in the conclusion of Lemma 
\ref{unique:support}.

We shall call a space $X$ {\em \B\/} provided that every compact subset of $X$ is  countable. A motivation for this definition comes from the classical notion of a Bernstein subset in the real line. One can easily see that a subset $X$ of the real line $\mathbb{R}$ is a Bernstein set if and only if both $X$ and its complement $\mathbb{R}\setminus X$ are \B\ spaces, in our terminology.

\begin{lemma}
\label{lemma:Lusin}
Assume that $\V$ is a variety of groups and $X$ is a $\V$-independent subset  of a separable metric group $K$ such that $|X|=\cont$.
Then there exists $X'\subseteq X$ such that $|X'|=\cont$ and $\grp{X'}$ is \B. 
\end{lemma}
\begin{proof}
Since $X$ is $\V$-independent, $\grp{X}$ is isomorphic to $F_\V(X)$,  so we can use the notation $\supp{g}{X}$ for all $g\in\grp{X}$. Since $K$ is separable metric, the family 
$$
\mathcal{C}=\{C\subseteq \grp{X}:
C
\mbox{ is compact and }
|C|=\cont\}
$$
has size at most $\cont$, so  we can fix an enumeration $\mathcal{C}=\{C_\alpha:\alpha<\cont\}$ of $\mathcal{C}$. By transfinite recursion on  $\alpha<\cont$ we shall choose $x_\alpha, y_\alpha\in X$ satisfying conditions (i$_\alpha$)--(iii$_\alpha$) below.
\begin{itemize}
\item[(i$_\alpha$)] $x_\alpha\not\in \{x_\beta:\beta<\alpha\}$,
\item[(ii$_\alpha$)] $\{x_\beta:\beta\le\alpha\}\cap \{y_\beta:\beta\le\alpha\}=\emptyset$,
\item[(iii$_\alpha$)] $y_\alpha\in \supp{g_\alpha}{X}$ for some $g_\alpha\in C_\alpha$.
\end{itemize}
\smallskip
{\sl Basis of recursion\/}.
Let $g_0\in C_0\setminus \{e\}$. Choose arbitrary $y_0\in\supp{g_0}{X}$ and $x_0\in X\setminus\{y_0\}$. Now conditions (i$_0$)--(iii$_0$) are satisfied.

\smallskip
{\sl Recursive step\/}.
Suppose that $\alpha<\cont$ and $x_\beta,y_\beta\in X$ were already chosen for all $\beta<\alpha$ so that conditions (i$_\beta$)--(iii$_\beta$) are satisfied. We shall choose $x_\alpha, y_\alpha\in X$ satisfying conditions (i$_\alpha$)--(iii$_\alpha$). Let 
\begin{equation}
\label{H:alpha}
H_\alpha=\grp{\{x_\beta:\beta<\alpha\}\cup\{y_\beta:\beta<\alpha\}}.
\end{equation}
Then $|H_\alpha|\le |\alpha|\cdot\omega<\cont$, Since $|C_\alpha|=\cont$, we can  choose 
\begin{equation}
\label{eq:g:alpha}
g_\alpha\in C_\alpha\setminus H_\alpha.
\end{equation}
 From \eqref{H:alpha} and \eqref{eq:g:alpha} it follows that $\supp{g_\alpha}{X}\not\subseteq \{x_\beta:\beta<\alpha\}$, so we can choose 
\begin{equation}
\label{eq:y:alpha}
y_\alpha\in \supp{g_\alpha}{X}\setminus \{x_\beta:\beta<\alpha\}.
\end{equation}
 From \eqref{eq:g:alpha} and \eqref{eq:y:alpha} we conclude  that (iii$_\alpha$) holds. Since $|X|=\cont$ and $|H_\alpha|<\cont$, we can choose 
\begin{equation}
\label{eq:x:alpha}
x_\alpha\in X\setminus (H_\alpha\cup\{y_\alpha\}).
\end{equation}
Now (i$_\alpha$) is satisfied by \eqref{H:alpha} and \eqref{eq:x:alpha}. It remains only to check condition (ii$_\alpha$).
Since (ii$_\beta$) holds for every $\beta<\alpha$,  we have $\{x_\beta:\beta<\alpha\}\cap\{y_\beta:\beta<\alpha\}=\emptyset$. Combining this with 
\eqref{eq:y:alpha} and \eqref{eq:x:alpha}, we get (ii$_\alpha$).

The recursive construction being complete, we claim that $X'=\{x_\alpha:\alpha<\cont\}\subseteq X$ is as required. Since (i$_\alpha$) holds for every $\alpha<\cont$, we have $|X'|=\cont$. Since (ii$_\alpha$) holds for every $\alpha<\cont$, for $Y=\{y_\alpha:\alpha<\cont\}$ we have $X'\cap Y=\emptyset$. 

It remains only to show that $\grp{X'}$  contains no uncountable compact subsets. Indeed, suppose that $C$ is an uncountable compact subset of $\grp{X'}$.  By \cite[Exercise 1.7.11]{E}, every  separable metric space is a union of a perfect set and a countable set. Since a perfect set has size $\cont$,
it follows that $|C|=\cont$. Since $C\subseteq \grp{X'}\subseteq \grp{X}$, we obtain $C\in\mathcal{C}$, and so  $C=C_\alpha$ for some $\alpha<\cont$. From (iii$_\alpha$),  there exists $g_\alpha\in C_\alpha$ such that $y_\alpha\in \supp{g_\alpha}{X}$. Since $y_\alpha\in Y$ and $Y\cap X'=\emptyset$, we conclude that $y_\alpha\not\in X'$. Therefore, $y_\alpha\in\supp{g_\alpha}{X}\setminus X'$. Since $X'\subseteq X$, this means that $g_\alpha\not\in\grp{X'}$.
We obtained a contradiction with $g_\alpha\in C_\alpha=C\subseteq \grp{X'}$.
\end{proof}

\begin{lemma}
\label{disturbing:lemma}
Let  $\V$ be a variety of groups and let $I$ be a set with $\omega_1\le|I|\le\cont$. Assume that $K$ is a  compact metric group, $X\subseteq K^I$ and $\varphi:X\to K$ is an injection such that:
\begin{itemize}
\item[(i)] $\varphi(X)$ is $\V$-independent,
\item[(ii)] $\grp{\varphi(X)}$ is \B,
\item[(iii)] $\grp{X}\in\V$,
\item[(iv)] for every $x\in X$ there exists  $J_x\in[I]^{\le\omega}$ such that  $\pi_i(x)=\varphi(x)$ for each $i\in I\setminus J_x$, where $\pi_i: K^I\to K$ is the projection on $i$th coordinate.
\end{itemize} 
Then $X$ is $\V$-independent and  $\grp{X}$ is \B.
\end{lemma}

\begin{proof}
 From (iv) one immediately gets the following claim.
\begin{Claim}
\label{cl:0}
For every $Y\in[X]^{\le\omega}$, the following holds:
\begin{itemize}
  \item[(a)] the set $I_Y=I\setminus \bigcup_{x\in Y} J_x$ is uncountable;
  \item[(b)] $\pi_i\restriction_Y=\varphi\restriction_Y$ for all $i\in I_Y$.
\end{itemize}
\end{Claim}

Let $Y$ be a finite subset of $X$.  Since $\grp{Y}\subseteq \grp{X}\in\V$ by (iii), it follows that $\grp{Y}\in\V$. By Claim \ref{cl:0}~(a), we can choose $i\in I_Y$. By Claim \ref{cl:0}~(b), $\pi_i\restriction_Y=\varphi\restriction_Y$. Since $\varphi$ is an injection, $\pi_i\restriction_Y$ is an injection as well. Since $\pi_i(Y)=\varphi(Y)\subseteq \varphi(X)$ and  $\varphi(X)$ is $\V$-independent by (i), we conclude that $Y$ is $\V$-independent \cite[Lemma 2.4]{DS-Mem}.
Since this holds for every finite subset $Y$ of $X$, it follows that $X$ is $\V$-independent \cite[Lemma 2.3]{DS-Mem}.

Since $X$ and $\varphi(X)$ are both $\V$-independent, there exists a unique isomorphism $\Phi:\grp{X}\to\grp{\varphi(X)}$ extending $\varphi$. The next claim is immediate from Claim \ref{cl:0}~(b) and our definition of $\Phi$.

\begin{Claim}
\label{cl:2}
For every $Y\in[X]^{\le\omega}$ one has $\pi_i\restriction_{\grp{Y}}=\Phi\restriction_{\grp{Y}}$ for all $i\in I_Y$.
\end{Claim}

For every subset $J$ of $I$ let $p_J:K^I\to K^J$ denote the projection.

Assume that $C$ is an uncountable compact subset of $\grp{X}$. Then $\Phi(C)$ is an  uncountable subset of $\grp{\varphi(X)}$, so the closure $F$ of $\Phi(C)$
is an uncountable compact subset of $K$. By (ii), $F\setminus \grp{\varphi(X)}\not=\emptyset$, so we can choose $b\in F\setminus \grp{\varphi(X)}\subseteq F\setminus \Phi(C)$.
Since $K$ is a metric space, $b\in F\setminus \Phi(C)$ and $\Phi(C)$ is dense in $F$, we can choose a faithfully indexed sequence $\{c_n:n\in \N\}\subseteq C$ such that the sequence $\{\Phi(c_n):n\in\N\}$ converges to $b$ in $K$.  Fix $Y\in [X]^{\le\omega}$ such that $\{c_n:n\in \N\}\subseteq \grp{Y}$. From Claim \ref{cl:2} we conclude that 
\begin{equation}
\label{projections:outside:I_Y}
\{\pi_i(c_n):n\in \N\}=\{\Phi(c_n):n\in \N\}
\mbox{ for all }
i\in I_Y.
\end{equation}

Use Claim \ref{cl:0}~(a) to fix $j\in I_Y$. Since the sequence $\{c_n:n\in \N\}$ is faithfully indexed and $\Phi$ is an injection, it follows from \eqref{projections:outside:I_Y}
that the sequence $\{\pi_j(c_n):n\in\N\}$ is faithfully indexed. Therefore, the sequence $\{p_S(c_n):n\in\N\}$ is faithfully indexed as well, where $S=\{j\}\cup\bigcup_{x\in Y} J_x$. Since $K^S$ is compact, the sequence $\{p_S(c_n):n\in\N\}$ has an accumulation point $y\in K^S$.  Define $g\in K^I$ by 
\begin{equation}
\label{eq:f:new}
g(i)=
\begin{cases}
y(i)
&
\mbox{if }
i\in S
\\
b
& 
\mbox{if }
i\in I\setminus S
\end{cases}
\hskip30pt
\mbox{ for all }
i\in I.
\end{equation}
\begin{Claim}
\label{cl:3}
$g$ belongs to the closure of the set $\{c_n:n\in\N\}$ in $K^I$. 
\end{Claim}

\begin{proof}
Let $W$ be an open neighbourhood of $g$ in $K^I$.  Then there exist an open set $U\subseteq K^S$ and an open set $V\subseteq K^{I\setminus S}$
such that $g\in U\times V\subseteq W$. Since $I\setminus S\subseteq I_Y$ and  the sequence $\{\Phi(c_n):n\in \N\}$ converges to $b$ in $K$, applying
\eqref{projections:outside:I_Y} and \eqref{eq:f:new} we can find $n_0\in\N$ such that $p_{I\setminus S}(c_n)\in V$ for all $n\in\N$ with $n\ge n_0$.
Since $y$ is an accumulation point of $\{p_S(c_n):n\in\N\}$, there exists 
an integer $m\ge n_0$ such that $p_S(c_m)\in U$. Now $c_m\in U\times V\subseteq W$. 
\end{proof}

Since $C$ is compact, it is closed in $K^I$. From $\{c_n:n\in\N\}\subseteq C$ and  Claim \ref{cl:3} we get $g\in C$. Since $C\subseteq \grp{X}$, it follows that $g\in \grp{X}$.
Let $E$ be a finite subset of $X$ with $g\in\grp{E}$. Since $I_{E}$ is uncountable by Claim \ref{cl:0}~(a) and $S$ is countable, we can choose $i\in I_{E}\setminus S$.
Then $b=\pi_i(g)=\Phi(g)$ by \eqref{eq:f:new} and Claim \ref{cl:2}. Thus, $b= \Phi(g)\in\Phi(\grp{X})=\grp{\varphi(X)}$, in contradiction with our choice of $b$. 
This proves that all compact subsets of $\grp{X}$ are countable.
\end{proof}

\begin{lemma}
\label{psc:Lusin}
Let  $\V$ be a variety of groups and let $I$ be a set with $\omega_1\le|I|\le\cont$.
Assume that $K\in\V$ is a compact metric group and $Z$ is a $\V$-independent subset of $K$ such that $|Z|=\cont$ and $\grp{Z}$
is \B. Then there exists a subset $X$ of $H=K^I$ with the following properties:
\begin{itemize}
\item[(a)] $X$ is a $\V$-independent subset of $H$ of size $\cont$;
\item[(b)] $\grp{X}$ is \B;
\item[(c)] $X$ is $G_\delta$-dense in $H$.
\end{itemize}
\end{lemma}
\begin{proof}
For every $J\in [I]^{\le\omega}$ let $K^J=\{y_{\alpha,J}:\alpha<\cont\}$ be an enumeration of $K^J$.

 From $|I|\le\cont$ it follows that $\left|[I]^{\le\omega}\right|\le\cont$, so we can fix a faithful enumeration $Z=\{z_{\alpha,J}:\alpha<\cont,J\in [I]^{\le\omega}\}$ of $Z$.

For $\alpha<\cont$ and $J\in [I]^{\le\omega}$ define $x_{\alpha,J}\in H$ by 
\begin{equation}
\label{eq:z}
x_{\alpha,J}(i)=
\begin{cases}
y_{\alpha,J}(i)
& 
\mbox{if }
i\in J
\\
z_{\alpha,J}
&
\mbox{if }
i\in I\setminus J
\end{cases}
\hskip30pt
\mbox{ for all }
i\in I.
\end{equation}

We claim that $X=\{x_{\alpha,J}:\alpha<\cont,J\in [I]^{\le\omega}\}$ has the desired properties. Define the bijection $\varphi:X\to Z$ by $\varphi(x_{\alpha,J})=z_{\alpha,J}$ for $(\alpha,J)\in\cont\times [I]^{\le\omega}$. Then items (i), (ii) and (iv)  of Lemma \ref{disturbing:lemma} are satisfied. Since $\grp{X}$ is a subgroup of $H=K^I$ and $K\in\V$, it follows that
$\grp{X}\in\V$, so item (iii) of Lemma \ref{disturbing:lemma} is satisfied as well. Applying this lemma, we conclude that $X$ is $\V$-independent and (b) holds.
Since $\varphi:X\to Z$ is a bijection, $|X|=|Z|=\cont$. Thus, (a) also holds.

It remains only to check (c). To achieve this, it suffices to show that  $\pi_J(\grp{X})=K^J$ for every $J\in [I]^{\le\omega}$, where $\pi_J:K^I\to K^J$ is the projection.  Fix such a $J$.
Let $y\in K^J$. There exists $\alpha<\cont$ such that $y=y_{\alpha,J}$. Now $\pi_J(x_{\alpha,J})=y_{\alpha,J}=y$ by \eqref{eq:z}. Since $x_{\alpha,J}\in X$, we are done.
\end{proof}

\medskip
\noindent
{\bf Proof of Theorem \ref{variety:theorem}:}
Let $K=L^\omega$. Then $K\in\V$ and  $K$ contains a $\V$-independent set of size $\cont$ \cite[Lemma 4.1]{DS-Mem}. Therefore, $K$  satisfies the assumptions of Lemma \ref{lemma:Lusin}. The conclusion of this lemma says that $K$ satisfies the assumptions of Lemma \ref{psc:Lusin}. Let $X$ be the set as in the conclusion of this lemma.
Then $D=\grp{X}$ is a $G_\delta$-dense subgroup of $K^I$ such that every compact subset $C$ of $D$ is countable; in particular, $w(C)\le |C|\le\omega$. 
Since $D$ is dense in $K^I$, we have $w(D)=w(K^I)=|I|\ge\omega_1$.  This shows that $D$ is not $w$-compact. 
Since $D=\grp{X}$ and $X$ is a $\V$-independent set of cardinality $\cont$, $D$ is a $\V$-free group with $\cont$ many generators. Note that $K^I\cong L^I$, as $I$ is uncountable.
\qed

\section{Proofs of Theorems \ref{ZFC} and \ref{CH}}
\label{sec:8}

The proof of the following well-known fact can be found, for example, in \cite[Theorem 4.15 and Discussion 4.14]{CRT1}.

\begin{fact}
\label{standard:images}
Let $G$ be a compact abelian group.
\begin{itemize}
\item[(i)] If $G$ is connected, then there exists a continuous surjective homomorphism of $G$ onto $\T^{w(G)}$.
\item[(ii)] If $\tau$ is a cardinal such that $\omega<\mathrm{cf}(\tau)\le\tau\le w(G)$, then there exists a continuous surjective homomorphism $f:G\to H=K^\tau$, where $K=\T$ or $K=\Z(p)$ for some prime number $p$.  
\end{itemize}
\end{fact}

The proof of the following fact can be found in \cite{HMbook}.

\begin{fact}
\label{weights}
If $N$ is a totally disconnected closed normal subgroup of a compact connected group $K$, then $w(K/N)=w(K)$.
\end{fact}

We denote by $G'$ the commutator subgroup of a group $G$. Recall that a group $G$ is {\em  perfect\/} if $G=G'$. A {\em semisimple\/} group is a perfect compact connected group \cite[Definition 9.5]{HMbook}. For a topological group $G$, we use $c(G)$ to denote the connected component of $G$ and we use $Z(G)$ for denoting the center of $G$. We need the following well-known fact.

\begin{fact}
\label{commutator:fact}
Let $G$ be a non-trivial compact connected group and let $A = c(Z(G))$. 
\begin{itemize}
 \item[(i)] $G=A \cdot G'$ and $\Delta= A\cap G'$ is totally disconnected;
 \item[(ii)] $G\cong (A\times G')/\Delta$ and $G/\Delta \cong A/\Delta \times G'/\Delta$;
 \item[(iii)] $w(G) = \max\{w(A), w(G')\}$;
 \item[(iv)] $w(A)= w(A/\Delta) = w(G/G')$;
 \item[(v)] if $G= G'$ is semisimple, then $A=\Delta= \{e\}$, $G/Z(G)$ is a product of compact simple Lie groups and $w(G/Z(G))=w(G)$; 
 \item[(vi)] the group  $G/\Delta$ admits a continuous surjective homomorphism onto $\T^{w(A)} \times \prod_{i\in I}L_i$, where each $L_i$ is a compact  simple Lie group and $w(G') = \omega\cdot |I|$; 
 \item[(vii)] if $\mathrm{cf}(w(G))> \omega$, then $G$ admits a  continuous surjective homomorphism onto $\T^{w(G)}$ or onto $L^{w(G)}$, for some compact  simple  Lie group $L$. 
\end{itemize}
\end{fact}

\begin{proof}
(i) This can be found in \cite[Theorem 9.24]{HMbook}.

(ii) Since $A$ is a central subgroup of $G$, the map $ f: A \times G'\to G$ defined by $f(a,g)=a^{-1}g$ for $(a,g)\in A \times G'$, is a continuous group homomorphism.
Clearly, $f$ is surjective. Since $\ker f = \Delta^*=\{(x,x):x\in\Delta\}\subseteq  A\times G'$ and $\Delta^*\cong\Delta$, we conclude that 
$$
G \cong   (A \times G')/\ker f=(A \times G')/\Delta^*
\cong 
(A \times G')/\Delta.
$$ 
Moreover, since 
$(\Delta \times \Delta)/\Delta^*=(\Delta \times \Delta)/\ker f=f(\Delta \times \Delta) = \Delta$, 
we obtain  
$$
A/\Delta \times G'/\Delta \cong  (A \times G')/(\Delta \times \Delta)\cong 
((A \times G')/\Delta^*)/((\Delta \times \Delta)/\Delta^*)
\cong G/\Delta.
$$

(iii)
 From (i) it follows that $G$ is a continuous image of $A \times G'$, so $w(G)\le w(A \times G')=\max\{w(A),w(G')\}$. Since both $A$ and $G'$ are subgroups of $G$, $\max\{w(A),w(G')\}\le w(G)$. This establishes 
(iii).

(iv)
 Since $A$ is connected, the first equality follows from (i) and Fact \ref{weights}. 
 From (i) one easily gets the  isomorphism  $G/G' \cong A/\Delta$, which gives the second equality.

(v)
 This is a particular case of a theorem of Varopoulos \cite{V}. The equality $w(G/Z(G))=w(G)$ follows from Fact \ref{weights} since $Z(G)$ is totally disconnected \cite[Theorem 9.19]{HMbook}.

(vi) 
By  (iv) and Fact \ref{standard:images}(i),
the connected compact abelian group $A/\Delta $ admits a continuous surjective homomorphism onto $\T^{w(A)}$. 

Since  $\Delta\subseteq Z(G)\subseteq Z(G')$,  the group $G'/\Delta$  has $G'/Z(G')$  as  its quotient.  Since $G'$ is semisimple \cite[Corollary 9.6]{HMbook}, from this  and item  
(v)
it follows that $G'/\Delta$ admits a continuous surjective homomorphism  onto a product $\prod_{i\in I}L_i$, where each $L_i$ is a compact  simple Lie group and 
$w(G') = w(G'/Z(G'))=\omega \cdot |I|$. 

Since $G/\Delta \cong A/\Delta \times G'/\Delta$ by (ii), we get the conclusion of item (vi).

(vii) Follows from 
(iii), (vi)
and the fact that there are only countably many pairwise non-isomorphic (as topological groups) compact  simple Lie groups. 
\end{proof}

\medskip
\noindent
{\bf Proof of Theorem \ref{ZFC}:}
Suppose that $G$ is not metrizable. If $G$ is abelian, we can  use Fact \ref{standard:images}(ii) to find a continuous surjective homomorphism  $f:G\to H=L^{\omega_1}$, 
where $L$ is either $\T$ or $\Z(p)$ for some prime number $p$. If $G$ is connected, we first use
Fact \ref{cont:images:of:smaller:weight} to find a continuous homomorphism
of $G$ onto (compact connected) group of weight $\omega_1$, and then we apply Fact \ref{commutator:fact}~(vii) 
to find a continuous surjective homomorphism  $f:G\to H=L^{\omega_1}$,  where $L$ is either $\T$ or a compact simple Lie group. When $L$ is abelian, we apply 
Corollary \ref{semi-Bernstein:theorem}
with $I=\omega_1$ to get a subgroup $D$  of $H$ as in the conclusion of this corollary. When $L$ is a compact simple Lie group,  we apply Corollary
\ref{Lie:corollary} with $I=\omega_1$ to get a subgroup $D$  of $H$ as in the conclusion of this corollary.
In both cases, we use Fact \ref{pull:back} to conclude  that $D_1=f^{-1}(D)$ is a $G_\delta$-dense subgroup of $G$ that is not projectively $w$-compact. This contradicts the assumption of our theorem. Therefore, $G$ must be metrizable.
\qed
\medskip

\begin{lemma}
\label{claim:1}
Assume CH. If $K=\T$ or $K=\Z(p)$ for some prime number $p$, then $H=K^{\omega_1}$ has a dense countably compact subgroup $D$ without infinite compact subsets.
\end{lemma}

\begin{proof}
We consider two cases.

\smallskip
{\sl Case 1\/}. $K=\T$.
Tkachenko \cite{Tka} constructed a dense countably compact subgroup $D$ of $K^{\omega_1}$ such that $|D|=\cont=\omega_1$ and $D$ has no non-trivial convergent sequences. 

\smallskip
{\sl Case 2\/}.
{\em $K=\Z(p)$ for some prime number $p$\/}.
In this case we can argue as follows. Since CH implies Martin's Axiom MA, and the group $L=\Z(p)^\omega$ is compact (in the Tychonoff product topology),
by the implication (a)~$\to$~(c) of \cite[Theorem 3.9]{DT}, the group $L$ admits a countably compact group topology without non-trivial convergent sequences.
An analysis of this proof shows that this topology comes from a monomorphism  
$j:L\to \Z(p)^\cont$ such that $D=j(L)$ is a dense subgroup of $\Z(p)^\cont$.
Under CH, we conclude that $H=K^{\omega_1}$ has a dense countably compact subgroup $D$ without non-trivial convergent sequences.
\footnote{In case $p=2$, one can also make a recourse to an old result of Hajnal and Juh\'asz \cite{HJ} claiming the existence of a subgroup $D$ of $K^{\omega_1}$ that is an HFD set. Such $D$ is a dense countably compact subgroup of $K^{\omega_1}$ without infinite compact subsets.}

\smallskip
The rest of the proof is common for both cases. Suppose that $X$ is an infinite compact subset of $D$. Since $D$ has no non-trivial convergent sequences, $X$ does not have any point of countable character. Then $|X|\ge 2^{\omega_1}>\omega_1=\cont$ by the \v{C}ech-Pospi\v{s}il theorem. This contradicts the inequality $|X|\le |D|=\cont$. This proves that every compact subset $X$ of $D$ is finite.
\end{proof}

\medskip
\noindent
{\bf Proof of Theorem \ref{CH}:}
Suppose that $G$ is not metrizable. Use Fact \ref{standard:images}(ii) to find a continuous surjective homomorphism  $f:G\to H=K^{\omega_1}$, 
where $K$ is either $\T$ or $\Z(p)$ for some prime number $p$. Let $D$ be a dense countably compact subgroup of $H$ without infinite compact subsets 
constructed in Lemma \ref{claim:1}. Since $D$ is dense in $H$, $w(D)=w(H)=\omega_1$. This shows that $D$ is not 
$w$-compact. By Fact \ref{pull:back}, $D_1=f^{-1}(D)$ is a dense countably compact  subgroup of $G$ that is not projectively $w$-compact. This contradicts the assumption of our theorem. Therefore, $G$ must be metrizable.
\qed
\medskip

\section{Examples}
\label{sec:examples}

\begin{example}
\label{proj:AV:not:proj:W-compact}
{\em 
For every cardinal $\tau$ such that $\omega_1\le\tau\le\cont$, there exists a pseudocompact \sAV\ group $D$ of weight $\tau$ that is not 
$w$-compact. Furthermore, under CH, $D$ can be chosen to be even countably compact.\/} Indeed, let $K=\T$ or $\Z(p)$ for some prime number $p$. Apply Corollary
\ref{semi-Bernstein:theorem} to $L=K$ and  $I=\tau$ to  find a $G_\delta$-dense subgroup $D$ of $K^\tau$ such that all compact subsets of $D$ are 
countable; in particular, $D$ is not $w$-compact. By Theorem \ref{CR:theorem}, $D$ is pseudocompact.
Under CH, we can use Lemma \ref{claim:1} to choose $D$ to be even  countably compact.
Since $w(D)= w(K^\tau)=\tau\le\cont$, from Proposition \ref{small:psc:are:sAV} we conclude that $D$ is \sAV.
\end{example}

Recall that a subgroup $D$ of a topological abelian group $G$ is called {\em essential\/} in $G$ if $D\cap N=\{0\}$ implies $N=\{0\}$ for every  closed subgroup $N$ of $G$  {\cite{B,P,St}}. A topological group $G$ is called \emph{minimal\/}  if there exists no Hausdorff group topology on $G$ strictly coarser  than the topology of $G$. A dense subgroup $D$ of a compact abelian group $G$ is  minimal if and only if $D$ is essential in $G$ \cite{B,P,St}.

\begin{example}\label{ex:omega-bounded}  
{\em 
Let $p$ be a prime number and $\kappa$ be an infinite cardinal. Define $\tau=2^{2^{2^\kappa}}$. Then
 there exists a dense essential (=minimal) $\kappa$-bounded $w$-compact subgroup of $\Z(p^2)^\tau$  that  is not \sAV\/}.
Indeed, let $G=\Z(p^2)^\tau$ and  let $f:G\to G$ be the (continuous) map defined $f(g)=pg$ for $g\in G$. Let $H=f(G)$. Then $H\cong \Z(p)^{\tau}$.
 From Proposition \ref{ex1}(ii), we get a  dense $\kappa$-bounded subgroup $D$ of $H \cong\Z(p)^\tau$ without the  \AVp.
Applying Fact \ref{pull:back}, we conclude that  $D_1=f^{-1}(D)$ is a dense $\kappa$-bounded subgroup of $G$ that 
is not \sAV. Since $pG=\ker f$ is easily seen to be an essential subgroup of $G$,
from $\ker f\subseteq D_1$ it follows that $D_1$ is an essential subgroup of $G$.
Finally, note that $\ker f\cong Z(p)^\tau$ is a compact subset of $D_1$
such that $w(\ker f)=w (Z(p)^\tau)=\tau=w(G)=w(D_1)$, which shows  that $D_1$ is $w$-compact.
\end{example}

For an infinite cardinal $\sigma$, define $\log \sigma = \min\{\tau \geq \omega : \sigma \leq 2^\tau\}$.
Let $\beth_0 =\omega$,  and let $\beth_{\alpha + 1}=2^{{\beth_\alpha  }} $ for every ordinal $\alpha$ and $\beth_\beta= \sup\{\beth_\alpha : \alpha < \beta\}  $ for every limit ordinal $\beta>0$. 

\begin{example}
\label{psc:example}
{\em
Let $G$ be a compact  group of weight $\sigma > \omega$.
} 
\begin{itemize}
  \item[(i)] {\em If $\mathrm{cf}(\log \sigma) = \omega$ and  $\sigma = (\log \sigma)^+$, then every $G_\delta$-dense subgroup of $G$ has the \AVp.\/}
Indeed, by Theorem \ref{nonAV:Gdelta}, it suffices to show that $m(\sigma)\ge\sigma$.
It is known that $\log \sigma\le m(\sigma)$ and $\mathrm{cf}(m(\sigma)) > \omega$ \cite[Theorem 2.7]{CoRo}. Therefore,
$m(\sigma) > \log \sigma$ and $m(\sigma)\ge (\log\sigma)^+=\sigma$ by our hypothesis.
  \item[(ii)] {\em If $\alpha$ is an ordinal of countable cofinality and $\sigma=\beth_\alpha^+$, then all $G_\delta$-dense subgroups of $G$ have the  \AVp.\/}
Indeed, it suffices to check that $\sigma = \beth_{\alpha}^+$ satisfies the hypothesis of item  (i).
Obviously,
$\log \sigma = \beth_{\alpha}$, so $\mathrm{cf}(\log\sigma)=\mathrm{cf}(\beth_{\alpha})= \mathrm{cf}(\alpha) = \omega$
and $\sigma = \beth_{\alpha}^+=(\log \sigma)^+$.
\end{itemize}
\end{example}

Here is an alternative proof of  item (ii) of this example that makes no recourse to its item (i)
and the cardinal function  $m(-)$.  Assume that $D$ is  a $G_\delta$-dense subgroup of $G$ without
the \AVp. Then $|D| < w(D)= w(G)=\beth_{\alpha}^+$, so $|D| \leq \beth_{\alpha}$. Since $\beth_{\alpha}$ is strong limit and $\beth_{\alpha}^+ = w(D) \leq 2^{|D|}$, we deduce that $|D| = \beth_{\alpha}$. Therefore, $D$ is a pseudocompact group such that $|D|$ a strong limit cardinal of countable cofinality. This contradicts a well-known theorem of van Douwen \cite{vD}. 

\section{Final remarks and open questions}
\label{sec:problems}

\begin{remark}
While ``projectively $w$-compact'' and ``\sAV'' are different properties when restricted to a single group,  the equivalence of items (ii) and (iv) of 
Corollary \ref{huge:corollary}
shows that these two properties and the property ``determining the completion'' coincide when imposed uniformly on  {\em all\/} dense subgroups of a given compact abelian group. Similarly, while it is unclear whether ``determining the completion'' and ``projectively $w$-compact'' are different properties for any given group, the equivalence of items (iii) and (v) of 
Corollary \ref{huge:corollary}
shows that these two properties coincide when imposed uniformly on 
{\em all\/} $G_\delta$-dense subgroups of a given compact abelian group.
\end{remark}

Recall that a topological group $G$ is called {\em totally minimal\/} if all (Hausdorff) quotient groups of $G$ are minimal.

\begin{remark}
\label{totally:minimal:remark}
\begin{itemize}
\item[(i)]
In a forthcoming paper \cite{DS:2} we prove that {\em every dense totally minimal subgroup of a compact abelian group $G$ determines $G$\/}. 
This shows that, in contrast with the results in Section \ref{metrization:section}, a weaker form of ``determination'' asking all 
dense totally minimal subgroups of $G$ to determine $G$ imposes no restrictions whatsoever  on a compact abelian group $G$.
\item[(ii)]
In a forthcoming paper \cite{DS1}
we prove that {\em totally minimal abelian groups are projectively $w$-compact\/}. Therefore, the italicized statement in item (i) 
shows that the answer to Question \ref{question:w-compact} is positive for 
this (proper) subclass of the class of projectively $w$-compact groups. 
\end{itemize}
\end{remark}

\begin{question}\label{Ques:w-compact}
What can one say about a compact (abelian) group $G$ such that all dense subgroups of $G$ are $w$-compact?
\end{question}

 From Theorem \ref{strong:limit} and Diagram 1 it follows that $w(G)$ must be a strong limit cardinal, but we do not know if $G$ must be metrizable.

\begin{question}
What is the minimal weight $\sigma$ of an $\omega$-bounded  abelian group that is not \sAV? Is $\sigma=\cont^+$?
\end{question}

We only know that $\cont^+\le\sigma\le 2^{2^\cont}$. The first inequality 
follows from Proposition \ref{small:psc:are:sAV} and the second inequality 
follows from Example \ref{ex:omega-bounded} (with $\kappa=\omega$).

\begin{question} 
Does Theorem \ref{ZFC} hold for all compact groups?
\end{question}

\begin{question}
\label{omitting:CH}
Does Theorem \ref{CH} hold in ZFC? Does  the implication (vi)$\to$(i) of Corollary \ref{huge:corollary} hold in ZFC?
\end{question}

As an intermediate step to solving this question, one may also wonder if CH can be weakened to Martin's Axiom MA in  Theorem \ref{CH} and 
the implication (vi)$\to$(i) of Corollary \ref{huge:corollary}.

We conjecture that the following question has a negative answer (although we have no counter-example at hand): 

\begin{question}
If every $\omega$-bounded dense subgroup of a compact abelian group $G$ determines it, must $G$ be metrizable?
\end{question}

Here come the counterpart of Question \ref{Ques:w-compact} for  $G_\delta$-dense subgroups: 

\begin{question} Describe the compact (abelian) groups $G$ such that every $G_\delta$-dense subgroup of $G$ is 
$w$-compact. 
\end{question}

\begin{question}
Let $K=\T$ or $K=\Z(p)$ for some prime number $p$. In ZFC, does there exist  a dense countably compact subgroup $D$ of 
$K^{\omega_1}$  without uncountable compact subsets?
\end{question}

As one can see from the proof of Theorem \ref{ZFC}, a positive answer to this question for $K=\T$ {\em and\/} $K=\Z(p)$ for {\em all\/} $p\in\P$ would yield a
positive answer to Question  \ref{omitting:CH}.

\medskip
\noindent
{\bf Acknowledgment}: The authors would like to thank Professor A.~V.~\AV\ for helpful discussions.

\end{document}